\documentclass{amsart}
\usepackage{setspace}
\usepackage{a4}
\usepackage{amssymb,amsmath,amsthm,latexsym}
\usepackage{amsfonts}
\usepackage{amsfonts}
\usepackage{graphicx}
\usepackage{textcomp}
\usepackage{cite}
\newtheorem{theorem}{Theorem}[section]

\newtheorem{corollary}[theorem] {Corollary}
\newtheorem{definition}[theorem]{Definition}

\newtheorem{proposition}[theorem]{Proposition}

\title{This is the title}
\usepackage{amssymb}
\usepackage{amssymb}
\usepackage{amssymb}
\usepackage{amssymb}
\usepackage{amsmath}
\usepackage{tikz}
\usepackage{hyperref}
\usepackage{enumerate}
\usepackage{mathtools}
\usepackage{amsmath}
\usepackage{tikz}

\begin{document}
\begin{center}
{\bf{THE NONCOMMUTATIVE $\ell_1-\ell_2$ INEQUALITY FOR HILBERT C*-MODULES AND THE EXACT CONSTANT}}\\
K. MAHESH KRISHNA AND P. SAM JOHNSON  \\
Department of Mathematical and Computational Sciences\\ 
National Institute of Technology Karnataka (NITK), Surathkal\\
Mangaluru 575 025, India  \\
Emails: kmaheshak@gmail.com, sam@nitk.edu.in

Date: \today\\
\end{center}

\hrule
\vspace{0.5cm}
\textbf{Abstract}: Let $\mathcal{A}$ be a unital C*-algebra. Then the theory of Hilbert C*-modules tells that 
\begin{align*}
\sum_{i=1}^{n}(a_ia_i^*)^\frac{1}{2}\leq \sqrt{n} \left(\sum_{i=1}^{n}a_ia_i^*\right)^\frac{1}{2}, \quad \forall n \in \mathbb{N}, \forall a_1, \dots, a_n \in \mathcal{A}.
\end{align*}
By modifications of    arguments of Botelho-Andrade, Casazza, Cheng, and Tran given in 2019, for certain tuple $x=(a_1, \dots, a_n) \in \mathcal{A}^n$, we give a method to compute a positive element $c_x$ in the C*-algebra $\mathcal{A}$ such that the equality 
	\begin{align*}
\sum_{i=1}^{n}(a_ia_i^*)^\frac{1}{2}=c_x \sqrt{n} \left(\sum_{i=1}^{n}a_ia_i^*\right)^\frac{1}{2}.
\end{align*}
holds. We give an application for the integral of G. G. Kasparov. We also derive the formula for the exact constant for the continuous $\ell_1-\ell_2$ inequality.



\section{Introduction}
Let $\mathbb{K}=\mathbb{C}$ or $\mathbb{R}$ and  $x \in \mathbb{K}^n$. Universally known $\ell_1-\ell_2$ inequality for Hilbert spaces states that $\|x\|_1\leq \sqrt{n}\|x\|_2.$ In 2019, Botelho-Andrade, Casazza, Cheng, and Tran 
\cite{BOTELHOCASAZZACHENGTRAN} gave a characterization which allows to compute a constant $c_x$, for a given $x$ such that $\|x\|_1= c_x \sqrt{n}\|x\|_2.$ First we recall this result. 
\begin{definition}\cite{BOTELHOCASAZZACHENGTRAN}\label{CASAZZADEFINITION}
	A vector $x=\frac{1}{\sqrt{n}}(c_1, \dots, c_n) \in \mathbb{K}^n$ is said to be a constant modulus vector if $|c_i|=1$, for all $i=1,\dots, n$.	
\end{definition}	
\begin{theorem}\cite{BOTELHOCASAZZACHENGTRAN}\label{CASAZZARESULT}
Let  $x=(a_1, \dots, a_n) \in \mathbb{K}^n$. The following are equivalent.
\begin{enumerate}[\upshape(i)]
	\item We have 
	\begin{align*}
	\|x\|_1=\left(1-\frac{c_x}{2}\right) \sqrt{n} \|x\|_2.
	\end{align*}
	\item We have 
	\begin{align*}
	\sum_{i=1}^{n}\left|\frac{|a_i|}{\|x\|_2}-\frac{1}{\sqrt{n}}\right|^2=c_x.
	\end{align*}
	\item The infimum of the distance from $\frac{x}{\|x\|_2}$ to the constant modulus vector is $\sqrt{c_x}$.
\end{enumerate}
In particular, 
\begin{align*}
\|x\|_1\leq  \sqrt{s}\|x\|_2 \iff \left(1-\frac{c_x}{2}\right) \sqrt{n} \leq \sqrt{s} \iff 1-\frac{c_x}{2} \leq \sqrt{\frac{s}{n}}   .
\end{align*}	
\end{theorem}
Theorem \ref{CASAZZARESULT} says that as long as we have equality connecting one-norm and two-norm, the constant can be determined using two-norm and the dimension of space. Further, it also helps to find the distance between $\frac{x}{\|x\|_2}$ to certain types of vectors (constant modulus vectors). This result found uses in nonlinear diffusion and diffusion state distances  \cite{MAGGIONIMURPHY, COWENDEVKOTA}. A variation of Theorem \ref{CASAZZARESULT} which concerns subspaces is the following.
\begin{theorem}\cite{BOTELHOCASAZZACHENGTRAN}\label{CASAZZAPRO}
	Let  $W$ be a   subspace of $\mathbb{K}^n$ and let  $ P:\mathbb{K}^n \to W$ be onto orthogonal projection. Then the following are equivalent.	
	\begin{enumerate}[\upshape(i)]
		\item For every unit vector $x \in W$,
		$
		\|x\|_1\leq \left(1-\frac{c_x}{2}\right)\sqrt{n}.
		$
		\item The distance of any unit vector in $ W$ to any constant modulus vector $ x \in W$ is greater than or equal to $\sqrt{c_x}$.
		\item For every constant modulus vector $x \in W$,
		$
		\|Px\|_2\leq 1-\frac{c_x}{2}.
	$
	\end{enumerate}
\end{theorem}
We organized this paper as follows. In Section \ref{MODULES}, we obtain a  result (Theorem \ref{GENERAL}), which is similar to first two implications of Theorem \ref{CASAZZARESULT},  in the context of Hilbert C*-modules. A partial result is obtained (Proposition \ref{PREVIOUS}) which corresponds to (iii) in Theorem \ref{CASAZZARESULT}. In Section  \ref{CONTINUOUSSECION}, we derive results which are similar to Theorems \ref{CASAZZARESULT} and \ref{CASAZZAPRO}, namely  Theorems \ref{IMP} and \ref{PRETHM}, respectively, for the function space $\mathcal{L}^2(X)$ whenever $\mu(X)<\infty$.

 \section{The noncommutative $\ell_1-\ell_2$ inequality for Hilbert C*-modules and the exact constant}\label{MODULES}
 
 Let $\mathcal{A}$ be a unital C*-algebra. Then the space  $\mathcal{A}^n$ becomes (left) Hilbert C*-module over the C*-algebra $\mathcal{A}$ w.r.t. the inner product 
 \begin{align*}
 \langle x, y\rangle\coloneqq \sum_{i=1}^{n}a_ib_i^*, \quad \forall x=(a_1, \dots, a_n), y=(b_1, \dots, b_n) \in \mathcal{A}^n
 \end{align*}
 and the norm 
 \begin{align*}
 \|x\|\coloneqq \|\langle x, x\rangle\|^\frac{1}{2}=\left\|\sum_{i=1}^{n}a_ia_i^*\right\|^\frac{1}{2}, \quad \forall x=(a_1, \dots, a_n)  \in \mathcal{A}^n
 \end{align*}
 (see \cite{LANCE, PASCHKE} for Hilbert C*-modules). 
 Let $a_1, \dots, a_n \in \mathcal{A}$ and let 
 
 \begin{align*}
 x=((a_1a_1^*)^\frac{1}{2}, \dots, (a_na_n^*)^\frac{1}{2}), \quad y=(1, \dots,  1)\in \mathcal{A}^n.
 \end{align*}
 By applying the Cauchy-Schwarz inequality in Hilbert C*-modules (Proposition 1.1 in \cite{LANCE}) for this pair we get 
 \begin{align*}
 \left(\sum_{i=1}^{n}(a_ia_i^*)^\frac{1}{2}\right)^2\leq n \sum_{i=1}^{n}a_ia_i^*, \quad \forall n \in \mathbb{N}, \forall a_1, \dots, a_n \in \mathcal{A}.
 \end{align*}
 By taking C*-algebraic square root (see Theorem 1.4.11 in \cite{LIN})
 \begin{align}\label{ELL1TWO}
 \sum_{i=1}^{n}(a_ia_i^*)^\frac{1}{2}\leq \sqrt{n} \left(\sum_{i=1}^{n}a_ia_i^*\right)^\frac{1}{2}, \quad \forall n \in \mathbb{N}, \forall a_1, \dots, a_n \in \mathcal{A}.
 \end{align}
 
 We call the Inequality (\ref{ELL1TWO}) as the  noncommutative $\ell_1-\ell_2$ inequality for Hilbert C*-modules. A standard result in C*-algebra is that an element $a\in \mathcal{A}$ is positive if and only if $a=bb^*$ for some $b\in \mathcal{A}$. Thus Inequality (\ref{ELL1TWO}) can also be written as 
 \begin{align*}
 \sum_{i=1}^{n}a_ia_i^*\leq \sqrt{n} \left(\sum_{i=1}^{n}(a_ia_i^*)^2\right)^\frac{1}{2}, \quad \forall n \in \mathbb{N}, \forall a_1, \dots, a_n \in \mathcal{A}.
 \end{align*}
   Note that Inequality (\ref{ELL1TWO}) is the  $\ell_1-\ell_2$ inequality for Hilbert spaces whenever the C*-algebra is the field of scalars.

  \begin{theorem}\label{GENERAL}
  	Let  $x=(a_1, \dots, a_n) \in \mathcal{A}^n$ be such that $\langle x, x \rangle $ is invertible.  The following are equivalent.
  	\begin{enumerate}[\upshape(i)]
  		\item We have 
  			\begin{align*}
  		\left(\sum_{i=1}^{n}(a_ia_i^*)^\frac{1}{2}\right)\langle x, x \rangle^\frac{1}{2}+\langle x, x \rangle^\frac{1}{2}\sum_{i=1}^{n}(a_ia_i^*)^\frac{1}{2}=\sqrt{n}\langle x, x \rangle^\frac{1}{2}(2-c_x)\langle x, x \rangle^\frac{1}{2}.
  		\end{align*}
  		\item We have 
  		\begin{align*}
  		\sum_{i=1}^{n}\left(\langle x, x \rangle^\frac{-1}{2}(a_ia_i^*)^\frac{1}{2}-\frac{1}{\sqrt{n}}\right)\left(\langle x, x \rangle^\frac{-1}{2}(a_ia_i^*)^\frac{1}{2}-\frac{1}{\sqrt{n}}\right)^*=c_x.
  		\end{align*}
  	
  	\end{enumerate}
 \end{theorem}
  \begin{proof}
  We make expansion and see 
  	\begin{align*}
  	&\sum_{i=1}^{n}\left(\langle x, x \rangle^\frac{-1}{2}(a_ia_i^*)^\frac{1}{2}-\frac{1}{\sqrt{n}}\right)\left(\langle x, x \rangle^\frac{-1}{2}(a_ia_i^*)^\frac{1}{2}-\frac{1}{\sqrt{n}}\right)^*\\
  	&=\sum_{i=1}^{n}\left(\langle x, x \rangle^\frac{-1}{2}(a_ia_i^*)^\frac{1}{2}(a_ia_i^*)^\frac{1}{2}\langle x, x \rangle^\frac{-1}{2}+\frac{1}{n}-\frac{\langle x, x \rangle^\frac{-1}{2}(a_ia_i^*)^\frac{1}{2}}{\sqrt{n}}-\frac{(a_ia_i^*)^\frac{1}{2}\langle x, x \rangle^\frac{-1}{2}}{\sqrt{n}}\right)\\
  	&=\langle x, x \rangle^\frac{-1}{2}\left(\sum_{i=1}^{n}a_ia_i^*\right)\langle x, x \rangle^\frac{-1}{2}+1-\frac{\langle x, x \rangle^\frac{-1}{2}}{\sqrt{n}}\sum_{i=1}^{n}(a_ia_i^*)^\frac{1}{2}-\left(\sum_{i=1}^{n}(a_ia_i^*)^\frac{1}{2}\right)\frac{\langle x, x \rangle^\frac{-1}{2}}{\sqrt{n}}\\
  	&=2-\frac{\langle x, x \rangle^\frac{-1}{2}}{\sqrt{n}}\sum_{i=1}^{n}(a_ia_i^*)^\frac{1}{2}-\left(\sum_{i=1}^{n}(a_ia_i^*)^\frac{1}{2}\right)\frac{\langle x, x \rangle^\frac{-1}{2}}{\sqrt{n}}=c_x
  	\end{align*}
  	if and only if 
  	\begin{align*}
  	\frac{\langle x, x \rangle^\frac{-1}{2}}{\sqrt{n}}\sum_{i=1}^{n}(a_ia_i^*)^\frac{1}{2}+\left(\sum_{i=1}^{n}(a_ia_i^*)^\frac{1}{2}\right)\frac{\langle x, x \rangle^\frac{-1}{2}}{\sqrt{n}}=2-c_x
  	\end{align*}
  	if and only if 
  	
  	\begin{align*}
  	\left(\sum_{i=1}^{n}(a_ia_i^*)^\frac{1}{2}\right)\langle x, x \rangle^\frac{1}{2}+\langle x, x \rangle^\frac{1}{2}\sum_{i=1}^{n}(a_ia_i^*)^\frac{1}{2}=\sqrt{n}\langle x, x \rangle^\frac{1}{2}(2-c_x)\langle x, x \rangle^\frac{1}{2}.
  	\end{align*}
  \end{proof}
A particular case of Theorem \ref{GENERAL} which is very similar to Theorem \ref{CASAZZARESULT} is the following.
\begin{corollary}
Let  $x=(a_1, \dots, a_n) \in \mathcal{A}^n$ be such that $\langle x, x \rangle $ is invertible and commutes with $\sum_{i=1}^{n}(a_ia_i^*)^\frac{1}{2}$. The following are equivalent.
\begin{enumerate}[\upshape(i)]
	\item We have 
	\begin{align*}
	\sum_{i=1}^{n}(a_ia_i^*)^\frac{1}{2}=\left(1-\frac{c_x}{2}\right) \sqrt{n} \langle x, x \rangle^\frac{1}{2}=\sqrt{n} \langle x, x \rangle^\frac{1}{2}\left(1-\frac{c_x}{2}\right).
	\end{align*}
	\item We have 
	\begin{align*}
	\sum_{i=1}^{n}\left(\langle x, x \rangle^\frac{-1}{2}(a_ia_i^*)^\frac{1}{2}-\frac{1}{\sqrt{n}}\right)\left(\langle x, x \rangle^\frac{-1}{2}(a_ia_i^*)^\frac{1}{2}-\frac{1}{\sqrt{n}}\right)^*=c_x.
	\end{align*}
\end{enumerate}
In particular, 
\begin{align*}
\sum_{i=1}^{n}(a_ia_i^*)^\frac{1}{2}\leq \sqrt{s} \langle x, x \rangle^\frac{1}{2} \iff \left(1-\frac{c_x}{2}\right) \sqrt{n} \leq \sqrt{s} \iff 1-\frac{c_x}{2} \leq \sqrt{\frac{s}{n}}   .
\end{align*}
\end{corollary}
We next derive a result which gives one sided implication in Theorem \ref{CASAZZARESULT}. For this, we need to generalize Definition \ref{CASAZZADEFINITION}.
\begin{definition}
	A vector $x=\frac{1}{\sqrt{n}}(c_1, \dots, c_n) \in \mathcal{A}^n$ is said to be a constant modulus vector if $c_ic_i^*=1$, for all $i=1,\dots, n$.	
\end{definition}
Recall that an element $a$ in a unital C*-algebra $\mathcal{A}$ is said to be an isometry if $a^*a=1$. Thus a vector  is a constant modulus vector is adjoint of each of its coordinates is an isometry upto scalar.
  \begin{proposition}\label{PREVIOUS}
  	Let  $x=(a_1, \dots, a_n) \in \mathcal{A}^n$ be such that $a_ia_i^*$ is invertible for each $i$.  Define 
  	\begin{align*}
  	c_x \coloneqq 2-\frac{\langle x, x \rangle^\frac{-1}{2}}{\sqrt{n}}\sum_{i=1}^{n}(a_ia_i^*)^\frac{1}{2}-\left(\sum_{i=1}^{n}(a_ia_i^*)^\frac{1}{2}\right)\frac{\langle x, x \rangle^\frac{-1}{2}}{\sqrt{n}}.
  	\end{align*}
  	Then the infimum of the distance from $\langle x, x \rangle^\frac{-1}{2}x$ to the constant modulus vector is less than or equal to $\sqrt{\|c_x\|}$.
  \end{proposition}
  \begin{proof}
  	Note that the condition $a_ia_i^*$ is invertible for each $i$ implies that   $\langle x, x \rangle $ is invertible. Now consider the vector $\frac{1}{\sqrt{n}}((a_1a_1^*)^\frac{-1}{2}a_1, \dots, (a_na_n^*)^\frac{-1}{2}a_n)$, which is unit modulus. Using the definition of infimum and by an expansion we get 
  	
  	\begin{align*}
  	&\inf\left\{\|\langle x, x \rangle^\frac{-1}{2}x-y\|: y=\frac{1}{\sqrt{n}}(c_1, \dots, c_n) \in \mathcal{A}^n \text{ is constant modulus vector}\right\}\\
  	&~= \inf\left\{\left\|\sum_{i=1}^{n}\left(\langle x, x \rangle^\frac{-1}{2}a_i-\frac{c_i}{\sqrt{n}}\right)\left(\langle x, x \rangle^\frac{-1}{2}a_i-\frac{c_i}{\sqrt{n}}\right)^*\right\|^\frac{1}{2}:c_i \in \mathcal{A}, c_ic_i^*=1,i=1,\dots, n\right\}\\
  	&~\leq \left\|\sum_{i=1}^{n}\left(\langle x, x \rangle^\frac{-1}{2}a_i-\frac{(a_ia_i^*)^\frac{-1}{2}a_i}{\sqrt{n}}\right)\left(\langle x, x \rangle^\frac{-1}{2}a_i-\frac{(a_ia_i^*)^\frac{-1}{2}a_i}{\sqrt{n}}\right)^*\right\|^\frac{1}{2}\\
  	&~\leq\left\|\sum_{i=1}^{n}\langle x, x \rangle^\frac{-1}{2}a_ia_i^*\langle x, x \rangle^\frac{-1}{2}-\frac{\langle x, x \rangle^\frac{-1}{2}}{\sqrt{n}}\sum_{i=1}^{n}(a_ia_i^*)^\frac{1}{2}-\left(\sum_{i=1}^{n}(a_ia_i^*)^\frac{1}{2}\right)\frac{\langle x, x \rangle^\frac{-1}{2}}{\sqrt{n}}+1\right\|^\frac{1}{2}\\
  	&~=\left\|2-\frac{\langle x, x \rangle^\frac{-1}{2}}{\sqrt{n}}\sum_{i=1}^{n}(a_ia_i^*)^\frac{1}{2}-\left(\sum_{i=1}^{n}(a_ia_i^*)^\frac{1}{2}\right)\frac{\langle x, x \rangle^\frac{-1}{2}}{\sqrt{n}}\right\|^\frac{1}{2}=\|c_x\|^\frac{1}{2}.
  	 	\end{align*}
  		
  \end{proof}
Proposition \ref{PREVIOUS} and Theorem \ref{CASAZZARESULT} lead to the following question: Does converse of Proposition \ref{PREVIOUS} hold?  We see that when $n=1$, $c_x=0$ and hence converse holds. It is not known that for $n\geq2$. Next we derive a result which concerns the  $\ell_1-\ell_2$ inequality for submodules of Hilbert C*-modules. 
\begin{proposition}\label{NEXT}
Let $\mathcal{N}$	be a submodule of $\mathcal{A}^n$ and $x \in \mathcal{N}$ be a vector such that $\langle x, x \rangle =1$.
If the distance of  $x$ to the constant modulus vector is greater than or equal to $c_x$, then 
\begin{align*}
c_x \leq \left\|2-\frac{2}{\sqrt{n}}\sum_{i=1}^{n}(a_ia_i^*)^\frac{1}{2}\right\|^\frac{1}{2}.
\end{align*}
\end{proposition}
\begin{proof}
By doing a similar calculation as in the proof of Proposition \ref{PREVIOUS} we get that 
\begin{align*}
c_x &\leq \inf\left\{\|x-y\|: y=\frac{1}{\sqrt{n}}(c_1, \dots, c_n) \in \mathcal{A}^n \text{ is a constant modulus vector}\right\}\\
&= \inf\left\{\left\|\sum_{i=1}^{n}\left(a_i-\frac{c_i}{\sqrt{n}}\right)\left(a_i-\frac{c_i}{\sqrt{n}}\right)^*\right\|^\frac{1}{2}:c_i \in \mathcal{A},c_ic_i^*=1,i=1,\dots, n\right\}\\
&\leq \left\|2-\frac{2}{\sqrt{n}}\sum_{i=1}^{n}(a_ia_i^*)^\frac{1}{2}\right\|^\frac{1}{2}.
\end{align*}
\end{proof}
Again a look at Proposition \ref{NEXT} and Theorem 2.4 in  \cite{BOTELHOCASAZZACHENGTRAN} which lead to the following question: Does converse of Proposition \ref{NEXT} hold?\\

  In the spirit of Theorem 3.1 in \cite{BOTELHOCASAZZACHENGTRAN}, we next give an  application of the previous theorem. For this we need  some concepts.\\
  Let $G$ be a compact Lie group and $\mu$ be the left Haar measure on  $G$ such that $\mu(G)=1$ (see \cite{SEPANSKI}). If $f, g:G \to \mathcal{A}$ are continuous functions, then we define
  \begin{align*}
  \langle f, g \rangle \coloneqq \int_{G}f(x)g(x)^*\,d\mu(x),
  \end{align*}
  where the integral is in the sense of G. G. Kasparov
  (see 
  \cite{KASPAROV, MANUILOVTROITSKY}). Now we can state the result.
  \begin{theorem}
  Let $G$ be a compact Lie group, $\mu(G)=1$, $f:G \to \mathcal{A}$ be continuous,  $f(x)\geq 0$, $\forall x \in G$ and $\langle f, f \rangle=1$.	The following are equivalent.
  		\begin{enumerate}[\upshape(i)]
  		\item We have
  		$
  		\int_{G}f(x)\,d\mu(x)=1-\frac{c}{2}.
  		$
  		\item We have 
  	$
  		\langle f-1, f-1 \rangle =c.
  	$
  	\end{enumerate}
  \end{theorem}
\begin{proof}
Consider 	
$
4=\langle f-1, f-1 \rangle+\langle f+1, f+1\rangle
=\langle f-1, f-1 \rangle +1+1+2\int_{G}f(x)\,d\mu(x)
=\langle f-1, f-1 \rangle +2+2\int_{G}f(x)\,d\mu(x)
$
which implies $\langle f-1, f-1 \rangle=2-2\int_{G}f(x)\,d\mu(x).$ Conclusion follows by taking $c=\langle f-1, f-1 \rangle=2-2\int_{G}f(x)\,d\mu(x).$
\end{proof}
\section{Exact constant for the continuous $\ell_1-\ell_2$ inequality}\label{CONTINUOUSSECION}
Let $X$ be a measure space with finite measure. Continuous Cauchy-Schwarz inequality tells that $\|f\|_1\leq \sqrt{\mu(X)}\|f\|_2$. Given  $f\in \mathcal{L}^2(X)$, we now derive a method for the exact constant in the equality $\|f\|_1= c_f \sqrt{\mu(X)}\|f\|_2$. For this, we reform the Definition \ref{CASAZZADEFINITION}.
\begin{definition}\label{DEFCON}
	A function  $f\in \mathcal{L}^2(X)$ is said to be a constant modulus function if  $|f(x)|=\frac{1}{\sqrt{\mu(X)}}, \forall x \in X$.
\end{definition}
Definition \ref{DEFCON} says that a function is constant modulus function if its image lies in the circle of radius $\frac{1}{\sqrt{\mu(X)}}$, centered at origin.
\begin{theorem}\label{IMP}
For $f\in \mathcal{L}^2(X)$, the following are equivalent.
\begin{enumerate}[\upshape(i)]
	\item We have 
	\begin{align*}
	\|f\|_1= \left(1-\frac{c_f}{2}\right) \sqrt{\mu(X)} \|f\|_2.
	\end{align*}
	\item We have 
	\begin{align*}
	\int_{X}\left|\frac{|f(x)|}{\|f\|_2}-\frac{1}{\sqrt{\mu(X)}}\right|^2\,d\mu(x)=c_f.
	\end{align*}
	\item The infimum of the distance from $\frac{f}{\|f\|_2}$ to the constant modulus function is $\sqrt{c_f}$.
\end{enumerate}	
In particular, 
\begin{align*}
 \|f\|_1\leq \sqrt{s} \|f\|_2 \iff \left(1-\frac{c_f}{2}\right) \sqrt{\mu(X)} \leq \sqrt{s} \iff 1-\frac{c_f}{2} \leq \sqrt{\frac{s}{\mu(X)}}   .
\end{align*}
\end{theorem}
\begin{proof}
(i) $\iff$ (ii) Starting from the integral in (ii)	we see that 

\begin{align*}
\int_{X}\left|\frac{|f(x)|}{\|f\|_2}-\frac{1}{\sqrt{\mu(X)}}\right|^2\,d\mu(x)&=\frac{1}{\|f\|_2^2}\int_{X}|f(x)|^2\,d\mu(x)+\frac{1}{\mu(X)}\int_{X}\,d\mu(x)\\
&\quad-2\frac{1}{\|f\|_2\sqrt{\mu(X)}}\int_{X}|f(x)|\,d\mu(x)\\
&=2\left(1-\frac{1}{\|f\|_2\sqrt{\mu(X)}}\int_{X}|f(x)|\,d\mu(x)\right)\\
&=c_f
\end{align*}
if and only if 
\begin{align*}
\frac{1}{\|f\|_2\sqrt{\mu(X)}}\int_{X}|f(x)|\,d\mu(x)=1-\frac{c_f}{2}
\end{align*}
if and only if 
\begin{align*}
\int_{X}|f(x)|\,d\mu(x)=\left(1-\frac{c_f}{2}\right)\|f\|_2\sqrt{\mu(X)}.
\end{align*}
(i) $\iff$ (iii) This follows from the calculation 
\begin{align*}
&\inf\left\{\left\|\frac{f}{\|f\|_2}-g\right\|_2: g \in \mathcal{L}^2(X) \text{ is constant modulus function}\right\}\\
&~=\inf\left\{\left(\int_{X}\left|\frac{f(x)}{\|f\|_2}-g(x)\right|^2\,d\mu(x)\right)^\frac{1}{2}: g \in \mathcal{L}^2(X) \text{ is a constant modulus function}\right\}\\
&~=\inf\bigg\{\left(\int_{X}\left|\frac{f(x)}{\|f\|_2}\right|^2\,d\mu(x)+\int_{X}|g(x)|^2\,d\mu(x)-\frac{2}{\|f\|_2}\text{Re}\left(\int_{X}f(x)\overline{g(x)}\,d\mu(x)\right)\right)^\frac{1}{2}:\\
&\quad \quad g \in \mathcal{L}^2(X) \text{ is a constant modulus function}\bigg\}\\
&~=\inf\bigg\{\left(1+1-\frac{2}{\|f\|_2}\text{Re}\left(\int_{X}f(x)\overline{g(x)}\,d\mu(x)\right)\right)^\frac{1}{2}:
 g \in \mathcal{L}^2(X) \text{ is a constant modulus function}\bigg\}\\
 &~=\left(2-\frac{2}{\sqrt{\mu(X)}\|f\|_2}\int_{X}|f(x)|\,d\mu(x)\right)^\frac{1}{2}.
\end{align*}
\end{proof}
To obtain further results we need a result whose proof will  follow from the routine argument using Hilbert projection theorem.
\begin{theorem}\label{IMPROVEDPROJECTION}
Let  $\mathcal{K}$ be a closed subspace of a Hilbert space $\mathcal{H}$ and let  $ P:\mathcal{H} \to \mathcal{K}$ be onto orthogonal projection. Then for each $ h \in \mathcal{H}$ with $Ph\neq0$, $\frac{Ph}{\|Ph\|}$ is the closest unit vector in $\mathcal{K}$ to $h$.
\end{theorem}

We now use Theorem  \ref{IMPROVEDPROJECTION} to obtain relations between closed subspaces of $ \mathcal{L}^2(X)$ and continuouss $\ell_1-\ell_2$ inequality.
\begin{theorem}\label{PRETHM}
Let  $W$ be a  closed subspace of $ \mathcal{L}^2(X)$ and let  $ P:\mathcal{L}^2(X) \to W$ be onto orthogonal projection. Then the following are equivalent.	
	\begin{enumerate}[\upshape(i)]
		\item For every unit vector $f \in W$,
		$
		\|f\|_1\leq \left(1-\frac{c_f}{2}\right)\sqrt{\mu(X)}.
		$
		\item The distance of any unit vector in $ W$ to any constant modulus function $ f \in W$ is greater than or equal to $\sqrt{c_f}$.
		\item For every constant modulus function $f \in W$,
		$
		\|Pf\|_2\leq 1-\frac{c_f}{2}.
		$
	\end{enumerate}
\end{theorem}
\begin{proof}
(i) $\iff$ (ii)	We do a similar calculation as in the proof of Theorem \ref{IMP} and get 
	\begin{align*}
	&\inf\left\{\left\|f-g\right\|_2: g \in \mathcal{L}^2(X) \text{ is a constant modulus function}\right\}\\
	&~=\inf\left\{\left(\int_{X}\left|f(x)-g(x)\right|^2\,d\mu(x)\right)^\frac{1}{2}: g \in \mathcal{L}^2(X) \text{ is a constant modulus function}\right\}\\
	&~=\inf\bigg\{\left(\int_{X}\left|f(x)\right|^2\,d\mu(x)+\int_{X}|g(x)|^2\,d\mu(x)-2\text{Re}\left(\int_{X}f(x)\overline{g(x)}\,d\mu(x)\right)\right)^\frac{1}{2}:\\
	&\quad \quad g \in \mathcal{L}^2(X) \text{ is a constant modulus function}\bigg\}\\
	&~=\inf\bigg\{\left(1+1-2\text{Re}\left(\int_{X}f(x)\overline{g(x)}\,d\mu(x)\right)\right)^\frac{1}{2}:
	g \in \mathcal{L}^2(X) \text{ is a constant modulus function}\bigg\}\\
	&~=\left(2-\frac{2}{\sqrt{\mu(X)}}\int_{X}|f(x)|\,d\mu(x)\right)^\frac{1}{2}.
	\end{align*}
	Therefore \\
	\begin{align*}
	\sqrt{c_f } \leq \left(2-\frac{2}{\sqrt{\mu(X)}}\int_{X}|f(x)|\,d\mu(x)\right)^\frac{1}{2} 
	\end{align*}
	if and only if 
	\begin{align*}
	\|f\|_1=\int_{X}|f(x)|\,d\mu(x)\leq \left(1-\frac{c_f}{2}\right)\sqrt{\mu(X)} .
	\end{align*}
	(ii) $\iff$ (iii) Let $ f \in W$ be constant modulus function. 
In view of 	Theorem \ref{IMPROVEDPROJECTION} we calculate 
\begin{align*}
\left\|\frac{Pf}{\|Pf\|}-f\right\|^2&=1+1-\left \langle \frac{Pf}{\|Pf\|}, f\right\rangle -\left \langle f, \frac{Pf}{\|Pf\|}\right\rangle=2-\left \langle \frac{P^2f}{\|Pf\|}, f\right\rangle -\left \langle f, \frac{P^2f}{\|Pf\|}\right\rangle\\
&=2-2\|Pf\|.
\end{align*}
Therefore 
\begin{align*}
c_f \leq  \left\|\frac{Pf}{\|Pf\|}-f\right\|   \text{\quad if and only if \quad }  \|Pf\| \leq   1-\frac{c_f}{2}.
\end{align*}
\end{proof}

  \section{Acknowledgements}
  First author thanks National Institute of Technology (NITK) Surathkal for
  financial assistance.

 \bibliographystyle{plain}
 \bibliography{reference.bib}

\end{document}